\documentclass{article}
\usepackage{amsmath}
\usepackage{amssymb}
\usepackage{latexsym}
\usepackage{amsthm}
\usepackage{mathrsfs}
\usepackage{wasysym}
\usepackage{fancyhdr}
\usepackage{amsfonts}
\usepackage{amssymb}
\usepackage{epsfig}
\usepackage{epstopdf}

\newtheorem{thm}{Theorem}[section]

\newtheorem{lemma}[thm]{Lemma}

\newtheorem{prob}[thm]{Problem}
\newtheorem{obs}[thm]{Observation}

\begin{document}

\title{On Perfect Matchings in Matching Covered Graphs}

 \author{Jinghua He\thanks{School of Mathematics and Statistics, Lanzhou University, Lanzhou, Gansu 730000, China.},\, Erling Wei\thanks{School of Information, Renmin University of China, Beijing 100872, China. Partially supported by a grant from NSFC (No. 11401576).},\, Dong Ye\thanks{Department of Mathematical Sciences, Middle Tennessee State University,
Murfreesboro, TN 37132, USA. Partially supported by a grant from the Simons Foundation (No. 359516).} \  and Shaohui Zhai\thanks{School of Applied Mathematics, Xiamen University of Technology, Xiamen, Fujian 361024, China.}
}

\maketitle

\begin{abstract}
 Let $G$ be a matching-covered graph, i.e., every edge is contained in a perfect matching. An edge subset $X$ of $G$ is feasible if there exists two perfect matchings $M_1$ and $M_2$ such that $|M_1\cap X|\not\equiv |M_2\cap X| \pmod 2$. Lukot'ka and Rollov\'a proved that an edge subset $X$ of a regular bipartite graph is not feasible if and only if $X$ is switching-equivalent to $\emptyset$, and they further ask whether a non-feasible set of a regular graph of class 1 is always switching-equivalent to either $\emptyset$ or $E(G)$? Two edges of $G$ are equivalent to each other if a perfect matching $M$ of $G$ either contains both of them or contains none of them. An equivalent class of $G$ is an edge subset $K$ with at least two edges such that the edges of $K$ are mutually equivalent. An equivalent class is not a feasible set. Lov\'asz proved that an equivalent class of a brick has size 2. In this paper, we show that, for every integer $k\ge 3$, there exist infinitely many $k$-regular graphs of class 1 with an arbitrarily large equivalent class $K$ such that $K$ is not switching-equivalent to either $\emptyset$ or $E(G)$, which provides a negative answer to the problem proposed by Lukot'ka and Rollov\'a.  Further, we characterize bipartite graphs with equivalent class, and characterize matching-covered bipartite graphs of which every edge is removable.

\end{abstract}

\section{Introduction}

Let $G$ be a graph. A {\em perfect matching} of $G$ is a set of independent edges which covers all vertices of $G$. A graph with a perfect matching is called a {\em matchable graph}. A graph $G$ is {\em $k$-extendable} if $G$ has at least $2k+2$ vertices and, for any $k$ independent edges of $G$, there is a perfect matching containing them.  It has been shown by Plummer \cite{P} that a $k$-extendable graph is $(k+1)$-connected. A 1-extendable graph is also called {\em matching-covered}, or {\em coverable}.
A 2-extendable bipartite graph is called a {\em brace}. By the result of Plummer \cite{P}, a brace is a 3-connected bipartite graph. A {\em brick} is a 3-connected graph such that, for any two vertices $u$ and $v$, $G\backslash \{u,v\}$ has a perfect matching. It is not hard to see that a brick is matching-covered but not bipartite. Plummer \cite{P} proved that a 2-extendable graph is either a brace or a brick. But a brick is not necessarily 2-extendable.
A matching-covered graph can be decomposed into a family of bricks and braces by the Lov\'asz's Tight-Cut Decomposition \cite{L87}.


A set of edges $X$ of a matching-covered graph $G$ is {\em feasible} if $G$ has two perfect matchings $M_1$ and $M_2$ such that $|M_1\cap X|\not\equiv |M_2\cap X|\pmod 2$. Note that, every edge of $G$ is contained by some perfect matchings but avoid by others. So a single edge of a matching-covered graph forms a trivial feasible edge set.
On the other hand, if $X$ is an edge-cut of $G$, the parity of $X\cap M$ depends on the parities of the orders of components of $G\backslash X$ and hence $X$ is always non-feasible.

A matching-covered regular graph may have many distinct perfect matchings. It has been conjectured by Lov\'asz and Plummer \cite{LP} that every matching-covered regular graph has exponentially many perfect matchings, which has been verified by Schrijver \cite{AS} for regular bipartite graphs and by Esperet et. al. \cite{EKKKN} for cubic graphs.  As a matching covered regular graph has many perfect matchings, it seems reasonable to believe that non-feasible edge sets are rare. It can be determined in randomized polynomial time whether a given edge set is feasible or not by using a probabilistic algorithm for exact matching (cf. Section 3.3 in \cite{L95}). Lukot'ka and Rollov\'a \cite{LR} show that the feasible sets in cubic graphs could be used to show the existence of spanning bipartite qudrangulations (cf. \cite{NNO}) and certain cycle covers in signed cubic bipartite graphs \cite{LR}.

Let $v$ be a vertex of $G$ and $E(v)$ be the set of all edges incident with $v$.
For a given edge set $X$, the {\em switching-operation} of $X$ on $E(v)$ is to be defined as the symmetric difference of $E(v)$ and $X$, denoted by $E(v)\oplus X=(E(v)\cup X)\backslash (E(v)\cap X)$. As a perfect matching
always contains exactly one edge from $E(v)$, the symmetric difference $E(v)\oplus X$ is feasible if and only if $X$ is feasible. Two edge sets $X_1$ and $X_2$ are {\em switching-equivalent} if $X_1$ can be obtained from $X_2$ by a series of switching-operations and vice visa. For two switching-equivalent edge sets $X_1$ and $X_2$, by the definition of switching-operation, $X_1$ is feasible if and only if $X_2$ is feasible.

\begin{thm}[Lukot'ka and Rollov\'a, \cite{LR}]
Let $G$ be a regular bipartite graph and $X\subseteq E(G)$. Then $X$ is not feasible if and only if $X$ is switching-equivalent to $\emptyset$.
\end{thm}

Lukot'ka and Rollov\'a \cite{LR} found that the Petersen graph has a non-feasible edge set which is not switching-equivalent to either $\emptyset$ or $E(G)$, and believe that an easy characterization of feasible edge sets for regular non-bipartite graphs seems not possible. More  examples can be found in \cite{NNO}. But all of these examples are not 3-edge-colorable cubic graphs, which are so-called snarks. For regular nonbipartite graphs of class 1, Lukot'ka and Rollov\'a propose the following problem.

\begin{prob}[Lukot'ka and Rollov\'a, \cite{LR}]\label{prob}
Let $G$ be a regular graph of class 1 and let $X$ be a subset of edges of $G$. Is it true that $X$ is not feasible if and only if $X$ is switching-equivalent to either $\emptyset$ or $E(G)$?
\end{prob}

In this paper, we provide a negative answer to the above problem by showing the following result.

\begin{thm}\label{thm:main}
For any integer $k\ge 3$, there are infinitely many $k$-regular nonbipartite graphs of class 1 with a non-feasible set $X$ which is not switching-equivalent to either $\emptyset$ or $E(G)$.
\end{thm}


\begin{figure}[!hbtp] \refstepcounter{figure}\label{fig1}
\begin{center}
\includegraphics[scale=.8]{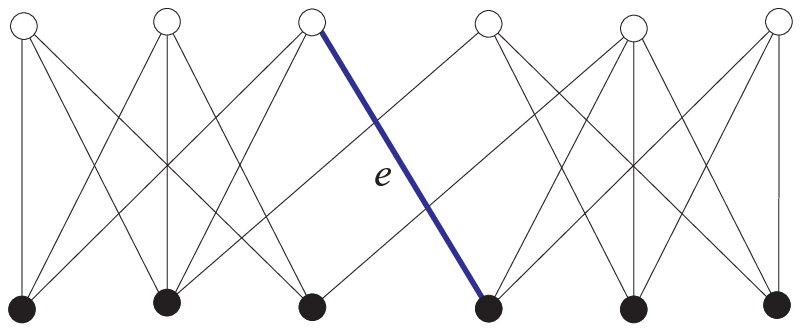}\medskip

{Figure~\ref{fig1}: A 3-connected bipartite graph with a non-removable edge $e$.}
\end{center}
\end{figure}

An edge $e$ of a matching-covered graph $G$ is {\em removable} if $G\backslash \{e\}$ is still matching-covered. A removable edge is also called a removable ear in Ear Decomposition of matching-covered graph \cite{CL, L83}, which provides a fundamental construction of matching-covered graphs \cite{CLM, L83, ZS} (see also \cite{LP}). A graph $G$ is {\em strongly coverable} if every edge of $G$ is removable. A strongly coverable graph is also called a graph with property $E(1,1)$ (cf. \cite{AHPP}). Note that a 2-extendable graph is strongly coverable \cite{P2}. Therefore, any two independent edges of a 2-extendable graph $G$ form a feasible set of $G$. Aldred et. al. \cite{AHPP} show that a strongly coverable bipartite graph is 3-connected. But a 3-connected bipartite graph is not necessarily strongly coverable. The bipartite graph in Figure~\ref{fig1} is 3-connected but not strongly coverable.

A matchable bipartite graph $G(A,B)$ is always balanced, i.e. $|A|=|B|$.  For two subsets $X$ and $Y$ of $V(G(A,B))$, let $E[X,Y]$ denote the set of all edges joining a vertex in $X$ and a vertex in $Y$. In this paper, we characterize all strongly coverable bipartite graphs as follows.

\begin{thm}\label{thm:main3}
Let $G(A,B)$ be a matching-covered bipartite graph. Then $G(A,B)$ is strongly coverable
 if and only if every edge-cut $S$ separating $G(A,B)$ into two balanced components $G_1(A_1,B_1)$ and $G_2(A_2,B_2)$  satisfies that $|E[A_1, B_2]|\ge 2$ and $|E[B_1,A_2]|\ge 2$.
\end{thm}

Two edges of a matching-covered graph $G$ are {\em equivalent} to each other if a perfect matching of $G$ either contains both of them or contains none of them. An {\em equivalent class} of $G$ is a subset of $E(G)$ with at least two edges such that any two edges of $K$ are equivalent to each other. An equivalent class of a matching-covered graph is not a feasible set. A matching-covered graph with an equivalent class $K$ is not strongly coverable because any edge of $K$ is not removable. However, a matching-covered graph without an equivalent class may not be strongly coverable, even for bipartite graphs. For example, the graph in Figure \ref{fig1} has no equivalent class but does have a non-removable edge $e$ and hence is not strongly coverable.


\begin{thm}[Lov\'asz, \cite{L87}]\label{thm:lovasz}
Let $G$ be a brick and $K$ be an equivalent class. Then $|K|=2$ and $G\backslash K$ is bipartite.
\end{thm}

In this paper, we obtain a characterization for bipartite graphs with an equivalent class as follows.

\begin{thm}\label{thm:main2} Let $G(A,B)$ be a matching-covered bipartite graph. Then $G(A,B)$ has an equivalent class if and only if $G(A,B)$ has a 2-edge-cut which separates $G(A,B)$ into two balanced components.
\end{thm}

The above result implies that a 3-connected matching-covered bipartite graph has no equivalent class. Therefore, a brace has no equivalent class. Together with Theorem~\ref{thm:lovasz}, a final graph in the Lov\'asz's Tight-Cut Decomposition either has no equivalent class or has an equivalent class of size two.

Let $\mathcal F_{\mbox{mc}}$, $\mathcal F_{\mbox{sc}}$, $\mathcal F_{\mbox{2-ext}}$ and $\mathcal F_{\mbox{nec}}$ denote the families of matching-covered graphs, strongly coverable graphs, 2-extendable graphs and graphs without equivalent class, respectively. Then we have the following nested relation:
\[\mathcal F_{\mbox{2-ext}}\subsetneq \mathcal F_{\mbox{sc}}\subsetneq \mathcal F_{\mbox{nec}}\subsetneq \mathcal F_{\mbox{mc}}.\]
In Section 2, we are going to prove Theorem~\ref{thm:main}. The proofs of Theorems~\ref{thm:main3} and \ref{thm:main2} are given in Section 3.

\section{Proof of Theorem~\ref{thm:main}}

 A signed graph $(G,\sigma)$ is a graph associated with a mapping $\sigma:E(G)\to \{-1,1\}$ which is called a signature. Let $E ^-(G,\sigma)=\{e \ | \sigma(e)=-1\}$.  Two signed graphs $(G,\sigma_1)$ and $(G,\sigma_2)$ are {\em switching-equivalent} if $E ^-(G,\sigma_1)$  is {\em switching-equivalent} to $E ^-(G,\sigma_2)$.
 A signed graph $(G,\sigma)$ is {\em balanced} if its negative edge set is {\em switching-equivalent} to the empty set.  For a subset $U\subseteq V(G)$, let $\nabla U$ denote the set of all edges joining a vertex in $U$ and a vertex in $V(G)\backslash U$. The following is a characterization of a balanced signed graph.

\begin{lemma}[Harary, \cite{FH}]\label{prop:equ}
A signed graph $(G,\sigma)$ is balanced if and only if  $E^-(G,\sigma)=\nabla U$ for some $U\subseteq V(G)$.
\end{lemma}

Let $G$ be a graph and $X\subseteq E(G)$. Define $\sigma_X: E(G)\to \{-1,1\}$ such that $\sigma_X(e)=-1$ if $e\in X$ and $\sigma_X(e)=1$ otherwise. Then we have a signed graph $(G,\sigma_X)$ for a graph $G$ and a given edge subset $X$. The following is a straightforward observation by applying the above lemma to signed graphs $(G,\sigma_X)$ and $(G,\sigma_{E(G)\backslash X})$.

\begin{obs}\label{obs:equ2}
Let $G$ be a graph and $X\subseteq E(G)$. Then $X$ is switching-equivalent to $\emptyset$ if and only if $X=\nabla  U$ for some $U\subseteq V(G)$; and $X$ is switching-equivalent to $E(G)$ if and only if $E(G)\backslash X=\nabla U$ for some $U\subseteq V(G)$.
\end{obs}

Now, we are going to prove our main result, Theorem~\ref{thm:main}.
\medskip

\noindent{\bf Proof of Theorem~\ref{thm:main}.}  For any integer $k\ge 3$, take a copy of the complete bipartite graph $K_{k,k}$. Assume that $(A, B)$ be the bipartition of $K_{k,k}$. The bipartite graph $K_{k,k}$ is $k$-edge-colorable and let $c: E(K_{k,k})\to \{1,...,k\}$ be a $k$-edge-coloring. Let $e_1=u_1v_1$ and $e_2=u_2v_2$ be two edges of $K_{k,k}$ with the same color, say $c(e_1)=c(e_2)=1$. Without loss of generality, assume that $\{u_1,u_2\} \subseteq A$ and $\{v_1,v_2\}\subseteq B$. Delete $e_1$ and $e_2$ from $K_{k,k}$ and let $G_k(A,B)$ be the resulting bipartite graph. Note that $G_k(A,B)$ has a Hamilton cycle.

Take $m$ copies of $G_k(A,B)$ ($m\ge 2$) and denote them by $G_k^1(A^1, B^1), G_k^2(A^2, B^2),..., G_k^m(A^m,B^m)$. Add the following edges to join these copies of $G_k$ to get a new $k$-regular non-bipartite graph $G(k,m)$:
\[u_1^1u_2^1,\; v_1^1u_1^2,\; v_2^1u_2^2,\; \cdots ,\; v_1^iu_1^{i+1},\; v_2^iu_2^{i+1},\; \cdots ,\; v_1^{m-1}u_1^m,\; v_2^{m-1}u_2^m,\; v_1^mv_2^m\] where $v_1^i, v_2^i, u_1^i, u_2^i\in V(G_k^i)$ with degree $k-1$. Let $K$ be the set of these new edges. For example, see $G(3,2)$ in Figure~\ref{fig:exp}.

\begin{figure}[!hbtp] \refstepcounter{figure}\label{fig:exp}
\begin{center}
\includegraphics[scale=.8]{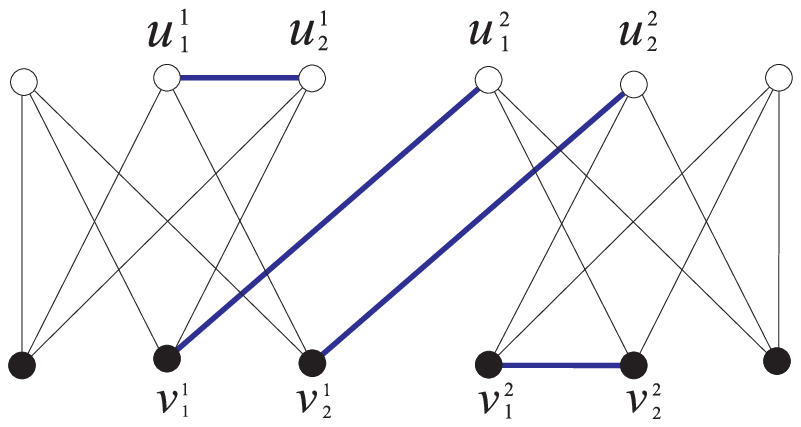}\medskip

{Figure~\ref{fig:exp}: An example $G(3,2)$: the set $K$ consisting of all blue edges.}
\end{center}
\end{figure}

Since $G_k (A,B)$ has a Hamiltonian cycle, the copy of $G_k^1(A^1,B^1)$ has a Hamilton cycle $C$ which together with $u_1^1u_2^1$ contains two odd cycles. Hence $G(k,m)$ is not bipartite. On the other hand, $G(k,m)$ has a $k$-edge-coloring which comes from a $k$-edge-coloring of $G_k$ together coloring all new edges by the color $c(e_1)=c(e_2)$. Hence $G(k,m)$ is a $k$-regular non-bipartite graph of class 1. Let $K$ be the set of all new edges.\medskip

\noindent{\bf Claim:} {\sl The edge set $K$ is an equivalent class of $G(k,m)$.} \medskip

\noindent{\em Proof of Claim.}  In the graph $G(k,m)$, two edges $v_1^iu_1^{i+1}$ and $ v_2^iu_2^{i+1}$ form a 2-edge-cut which separates $G(k,m)$ into two components with an even number of vertices. Hence a perfect matching of $G(k,m)$ contains either none of them or both of them. So $v_1^iu_1^{i+1}$ is equivalent to $v_2^iu_2^{i+1}$ for $i=1,...,m-1$.

Let $M$ be a perfect matching of $G(k,m)$ containing both $v_1^iu_1^{i+1}$ and $v_2^iu_2^{i+1}$. Consider the copy $G_k^{i+1}(A^{i+1}, B^{i+1})$. The perfect matching $M$ matches all vertices $A^{i+1}\backslash \{u_1^{i+1}, u_2^{i+1}\}$ to $k-2$ vertices of
$B^{i+1}$. So the remaining two vertices of $B^{i+1}$ are matched to two vertices of $A^{i+2}$ where $i+1\le m-1$. Hence $v_1^{i+1}u_1^{i+2}\in M$ and $v_2^{i+1}u_2^{i+2}\in M$. A similar argument shows that $K\subseteq M$. So all edges in $K$ are dependent on $v_j^iu_j^{i+1}$ for any $j\in \{1,2\}$ and $i\in \{1,...,m-1\}$, which implies that $K\backslash \{u_1^1u_1^2, v_1^mv_2^m\}$ is an equivalent class.

On the other hand, a perfect matching $M$ of $G(k,m)$ containing $u_1^1u_2^1$  matches $v_1^1$ and $v_2^1$ to $u_1^2$ and $u_2^2$ respectively. So all edges of $K$ are dependent on $u_1^1u_2^1$. By symmetry, all edges of $K$ are dependent on $v_1^mv_2^m$ too. It follows that $K$ is an equivalent class of $G(K,m)$. This completes the proof of Claim. \medskip

Let $X=\{u_1^1u_2^1, v_1^1u_1^2,...,v_1^iu_1^{i+1},...v_1^{m-1}u_1^m\}\subset K$. So $X$ is an equivalent class by Claim. Hence not a feasible set.  In the following, it suffices to show that $X$ is not switching-equivalent to either $\emptyset$ or $E(G(k,m))$.

First, note that $G(k,m)\backslash X$ is connected. Therefore, there is no $U\subseteq V(G(k,m))$ such that $X=\nabla U$. On the other hand, $G(k,m)\backslash X$ is not a bipartite graph because the edge $v_1^mv_2^m$ together with a Hamilton cycle $C$ of $G_k(A^m,B^m)$ contains two odd cycles which belong to $G(k,m)\backslash X$. So $G(k,m)$ does not have a vertex subset $U$ such that $E(G(k,m)\backslash X=\nabla U$. Hence $X$ is not switching-equivalent to $\emptyset$ or $E(G)$ by Observation~\ref{obs:equ2}. Hence $G(k,m)$ is a $k$-regular non-bipartite graph of class 1 which has a non-feasible set $X$ not switching-equivalent to $\emptyset$ or $E(G(k,m))$.

As $m\ge 2$ could be any integer, there are infinitely many such graphs $G(k,m)$ for any $k\ge 3$ with a non-feasible set which is not switching-equivalent to $\emptyset$ or  $E(G(k,m))$. This completes the proof of the theorem.
\qed \medskip

\noindent{\bf Remark.} In the above construction, the complete bipartite graph $K_{k,k}$ could be replaced by any $k$-regular bipartite graph $G$ with a Hamilton cycle $C$. For a $k$-edge-coloring of $G$, choose two edges with the same color but not from the cycle $C$ to be deleted. Let $G'$ be the resulting bipartite graph and then take $m$ copies of $G'$. Then the construction generates infinitely many other examples. \medskip

The graph $G(k,m)$ from the above construction is a matching-covered graph with an equivalent class of size $2m$. So the equivalent class of a matching-covered graph could goes to arbitrarily large. However, the edge-connectivity of $G(k,m)$ is 2. We do not know whether there are highly connected matching-covered graphs with a large equivalent class. Theorem~\ref{thm:lovasz} shows that bricks do not have  a large equivalent class. In the next section, we show that the edge-connectivity of a matching-covered bipartite graph $G$ is 2 if it has an equivalent class.

\section{Matchable bipartite graphs}

Let $G(A,B)$ be a matchable bipartite graph with bipartition $(A,B)$, and let $M$ be a perfect matching of $G(A,B)$.
A cycle $C$ of $G(A,B)$ is {\em $M$-alternating} if $E(C)\cap M$ is a perfect matching of $C$. Similarly, a path $P$ of $G(A,B)$ is {\em $M$-alternating} if $E(P)\cap M$ is a perfect matching of $P$. Hall's Theorem provides a characterization of matchable graph, which says that a bipartite graph $G(A,B)$ is matchable if and only if $|A|=|B|$ and for any $U\subseteq A$, $|N(U)|\ge |U|$. The following is a similar result for matching-covered bipartite graph.

\begin{lemma}[Theorem 4.1.1 in \cite{LP}]\label{lem:LP}
Let $G(A,B)$ be a bipartite graph. Then $G(A,B)$ is matching-covered if and only if $|A|=|B|$ and for any proper subset $U\ne \emptyset$ of $A$, $|N(U)|\ge |U|+1$.
\end{lemma}

Let $G(A,B)$ be a matching-covered graph. For any two vertex $x\in A$ and $y\in B$ such that $xy\notin E(G(A,B))$, $G\cup \{xy\}$ is matching-covered by Lemma~\ref{lem:LP}. Hence, $G\cup \{xy\}$ has a perfect matching $M$ containing $xy$, and another perfect matching $M'$ containing an edge of $G$ incident with $x$. Therefore, the symmetric difference $M\oplus M'$ has a cycle $C$ containing $xy$. Further, $G$ has an $M'$-alternating path joining $xy$, which is $C\backslash \{xy\}$. So the following lemma holds.

\begin{lemma}\label{lem:path}
Let $G(A,B)$ be a matching-covered bipartite graph. Then for any vertex $x\in A$ and $y\in B$, there is an $M$-alternating joining $x$ and $y$ for some perfect matching $M$.
\end{lemma}

For matchable bipartite graphs, the Dulmage-Mendelsohn Decomposition \cite{DM} provides a structure characterization as follows.

\begin{lemma}[Dulmage and Mendelsohn, \cite{DM}]\label{lem:DM-decom}
Let $G(A,B)$ be a matchable bipartite graph. Then $G(A,B)$ has a decomposition into disjoint matching-covered subgraphs $Q_1,...,Q_k$ such that:\\
(1)  every $Q_i$ is vertex induced and, \\
(2) for any $e\in E[Q_i, Q_j]$ with $i, j\in \{1,2,...,k\}$, $e$ is not contained by any perfect matching of $G$.
\end{lemma}

\begin{figure}[!hbtp] \refstepcounter{figure}\label{fig2}
\begin{center}
\includegraphics[scale=1]{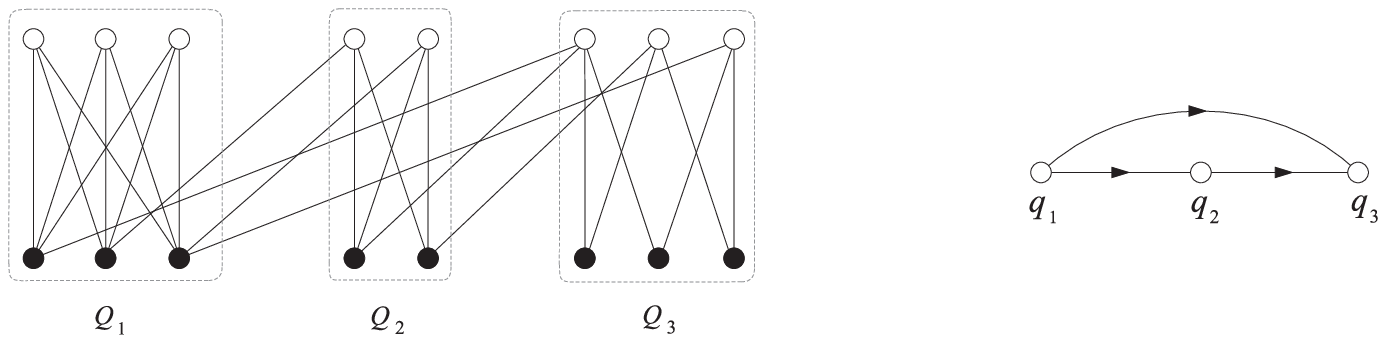}\medskip

{Figure~\ref{fig2}: The Dulmage-Mendelsohn Decomposition of $G(A,B)$ and the Dulmage-Mendelsohn digraph $D$ (right).}
\end{center}
\end{figure}

For a matchable bipartite graph, the Dulmage-Mendelsohn Decomposition is unique.  Let $G(A,B)$ be a matchable bipartite graph and let $\mathcal G=\{Q_1,...,Q_k\}$ be the  Dulmage-Mendelsohn Decomposition.  For any $1\le i\le k$, identify all vertices in $A\cap Q_i$ to a vertex $u_i$ and all vertices in $B\cap Q_i$ to a vertex $v_i$  and delete all multiple-edges to get a simple bipartite graph. For an edge $u_iv_j$, orient it from $u_i$ to $v_j$ if $i=j$ and from $v_j$ to $u_i$ if $i\ne j$. Since the Dulmage-Mendelsohn Decomposition is unique, the digraph generated this way is unique and is denoted by $D'$. The {\em Dulmage-Mendelsohn digraph} $D$ is obtained from $D'$ by contracting all arcs $u_iv_i$ to a single vertex $q_i$ for all $1 \le i\le k$. (For example, see Figure~\ref{fig2}.) So if $G(A,B)$ is matching-covered, then $\mathcal G$ has only one graph and hence $D$ has one vertex but no arcs. The following is a property of the Dulmage-Mendelsohn digraph $D$ of a matchable bipartite graph $G(A,B)$.

\begin{lemma}\label{lem:acyclic}
Let $G(A,B)$ be a connected matchable bipartite graph. If $G(A,B)$ is not matching-covered, then
the Dulmage-Mendelsohn digraph $D$ of  $G(A,B)$ is acyclic.
\end{lemma}
\begin{proof}
Let $G(A,B)$ be a matchable bipartite graph and let $\mathcal G=\{Q_1,...,Q_k\}$ be the  Dulmage-Mendelsohn Decomposition. Since $G(A,B)$ is not matching-covered, then $k\ge 2$. Let $D$ be the Dulmage-Mendelsohn digraph. Since $G(A,B)$ is connected, $D$ has at least one arc.
Suppose to the contrary that $D$ has a directed cycle $C$. Without loss of generality, assume that $C=q_1q_2\cdots q_mq_1$ for some $2\le m\le k$ (relabeling if necessary).

By the definition of $D$, for each arc $q_iq_{i+1}$ where $i$ and $i+1$ are taken modulo $m$, $G(A,B)$ has an edge joining a vertex $u_{i+1}\in Q_{i+1}\cap A$ and a vertex $v_i\in Q_i\cap B$ which is not contained by any perfect matching of $G(A,B)$ by (2) in Lemma~\ref{lem:DM-decom}. In each $Q_i$ with $1\le i\le m$, there exists an $M_i$-alternating path $P_i$ joining $u_i$ and $v_i$ for some perfect matching $M_i$ of $Q_i$ by Lemma~\ref{lem:path}. For $m+1
\le i \le k$, let $M_i$ be a perfect matching of $Q_i$ which is matching-covered. Let $M=\cup_{i=1}^k M_i$ and
\[C':=(\cup_{i=1}^m P_i)\cup \{v_iu_{i+1} | i, i+1\in \{1,...,m\} \pmod m\}.\] Then $M$ is a perfect matching of $G$ and $C'$ is an $M$-alternating cycle of $G$. So the symmetric difference $M\oplus E(C')$ is another perfect matching containing edges $v_iu_{i+1}$, which contradicts that $v_iu_{i+1}$ is not contained in any perfect matching of $G(A,B)$. This completes the proof.
\end{proof}

By Lemma~\ref{lem:acyclic} and the definition of the Dulmage-Mendelsohn digraph, if $D$ has an arc $q_iq_j$, then all edges of $E(Q_i,Q_j)$ join vertices of $Q_i\cap B$ and the vertices of $Q_j\cap A$. In other words,  $E[Q_i\cap A, Q_j\cap B]=\emptyset$. On the other hand, if $E[Q_i\cap B, Q_j\cap A]\ne \emptyset$, then $q_iq_j$ is an arc of $D$.

Let $G(A,B)$ be a matchable bipartite graph, but not matching-covered. Then, by Lemma~\ref{lem:acyclic}, the Dulmage-Mendelsohn digraph $D$ of $G(A,B)$ is acyclic.  A {\em directed cut} $S$ of $D$ is a subset of arcs of $D$ which separates $D$ into two components and all arcs of $S$ are oriented from the one component to the other. A family of directed paths $\mathcal P$ {\em intersects} all directed cuts of $D$ if for any directed cut $S$ of $D$, there exists a path $P\in \mathcal P$ such that $E(P)\cap S\ne \emptyset$.   The following result shows how many new edges should be added to a non-matching-covered bipartite graph to obtain a matching-covered bipartite graph.

\begin{thm}\label{thm:edge}
Let $G(A,B)$ be a matchable bipartite graph and let $G'(A,B)$ be a smallest matching-covered bipartite graph such that $G(A,B)\subseteq G'(A,B)$. Then \[ |E(G'(A,B))|\le |E(G(A,B))|+\ell,\] where $\ell$ is the smallest size of a family of directed paths intersecting all directed cuts of the Dulmage-Mendelsohn digraph $D$ of $G(A,B)$.
\end{thm}
\begin{proof}
Let $G(A,B)$ be a matchable bipartite graph  and let $D$ be the Dulmage-Mendelsohn digraph. If $G(A,B)$ is a matching-covered graph, then $D$ is a single vertex and $\mathcal P=\emptyset$. The theorem holds trivially. So in the following, assume that $G(A,B)$ is not matching-covered. Therefore, the Dulmage-Mendelsohn Decomposition $\mathcal G=\{Q_1,...,Q_k\}$ of $G(A,B)$ has at least two graphs, i.e., $k\ge 2$.  By Lemma~\ref{lem:acyclic}, $D$ is acyclic. Let $\mathcal P$ be a family of directed paths intersecting all directed cuts of $D$ such that $|\mathcal P|=\ell$.

For any $P\in \mathcal P$, add an arc $e_P$ from the terminal vertex of $P$ to the initial vertex of $P$, and let the new digraph be $D'$. Since $\mathcal P$ intersects all directed cuts of $D$, $D'$ has no directed cut and hence is strongly-connected. Hence, for any arc $e$ of $D$, $D'$ has a directed cycle containing $e$.

For each new arc $e_P=x_ix_j$, then add a new edge to $G$ joining a vertex $v_i\in B\cap Q_i$ and a vertex $u_j\in A\cap Q_j$. Let the new bipartite graph be $G'(A,B)$. Let $e$ be an edge of $G'(A,B)$. If $e$ is an edge of some $Q_i$, then $e$ is contained in a perfect matching of $G(A,B)$ which is also a perfect matching of $G'(A,B)$. If $e$ is an edge of $E[Q_i, Q_j]$, the digraph $D'$ has a directed cycle $C$ containing the arc  $q_iq_j$ or $q_jq_i$. By a similar argument as in Lemma~\ref{lem:acyclic}, the directed cycle $C$ of $D'$ corresponds to an $M$-alternating cycle in $G'(A,B)$ for some perfect matching $M$ of $G'(A,B)$. Therefore, $e$ is contained in a perfect matching of $G'(A,B)$. So $G'(A,B)$ is matching-covered. Hence, the number of edges of a smallest matching-covered graph containing $G(A,B)$ is at most $|E(G'(A,B))|=|E(G(A,B))|+|\mathcal P|=|E(G(A,B))|+\ell$.
\end{proof}

Now, we are going to prove our main results, Theorems~\ref{thm:main3} and \ref{thm:main2}.\medskip

\noindent{\bf Proof of Theorem~\ref{thm:main3}.} Let $G(A,B)$ be a matching-covered bipartite graph.

First, assume that $G(A,B)$ is strongly coverable. Let $S$ be an edge-cut of $G(A,B)$, which separates $G(A,B)$ into two balanced components $G_1(A_1,B_1)$ and $G_2(A_2,B_2)$. Then $S=E[A_1,B_2]\cup E[A_2,B_1]$.  We need to show that $|E[A_1,B_2]|\ge 2$ and $|E[B_1, A_2]|\ge 2$. If not, we may assume that $|E[A_1,B_2]|\le 1$ by symmetry. Let $e\in E[A_1, B_2]$. Then $G(A,B)\backslash e$ has no edges joining vertices of $A_1$ to vertices $B_2$. Since $G_1(A_1,B_1)$ is balanced, any perfect matching of $G(A,B)\backslash e$ does not contain edges from $E[B_1,A_2]$. Therefore, $G(A,B)\backslash e$ is not matching-covered. Hence $G(A,B)$ is not strongly coverable, a contradiction to the assumption that $G(A,B)$ is strongly coverable.

In the following, assume that every edge-cut $S$ separating $G(A,B)$ into two balanced components $G_1(A_1,B_1)$ and $G_2(A_2,B_2)$ satisfies $|E[A_1,B_2]|\ge 2$ and $|E[A_2,B_1]|\ge 2$. We need to show that $G(A,B)$ is strongly coverable. In other words, for any edge $e$, $G(A,B)\backslash e$ is matching-covered.
If not, then $G(A,B)$ has an edge $e$ such that $G(A,B)\backslash e$ is not matching-covered. Let $\mathcal G=\{Q_1, Q_2,...,Q_k\}$ be the Dulmage-Mendelsohn Decomposition of $G(A,B)\backslash e$, and let $D$ be the Dulmage-Mendelsohn digraph. By Lemma~\ref{lem:acyclic}, $D$ is a cyclic. By Theorem~\ref{thm:edge}, adding one more arc to $D$ generates a strongly connected digraph $D'$. Therefore, $D$ has only exactly one sink and one source. Without loss of generality, assume $q_1$ and $q_k$ be the source and sink of $D$, respectively, where $q_1$ and $q_k$ correspond to the graphs $Q_1$ and $Q_k$.  By the definition of $D$, all edges of $G(A,B)\backslash e$ joining vertices of $Q_1$ to vertices $Q_i$ with $i\ne 1$ are incident with vertices in $Q_1\cap B$. So the edge $e$ joins a vertex in $Q_1\cap A$ and a vertex in $Q_k\cap B$. Let $S=\nabla V(Q_1)$, the set of all edges joining vertices of $Q_1$ and vertices of its component in $G(A,B)$. Then $S$ is an edge-cut separating $G(A,B)$ into $G_1(A_1,B_1)=Q_1$ and $G_2(A_2,B_2)=G(A,B)\backslash Q_1$, where $A_1=V(Q_1)\cap A$ and $B_1=V(Q_1)\cap B$. Note that both $G_1(A_1,B_1)$ and $G_2(A_2,B_2)$ are matchable and therefore balanced. However, $|E(A_1, B_2)|=|\{e\}|=1$, a contradiction to the assumption. This completes the proof.
\qed
\bigskip

\noindent{\bf  Proof of Theorem~\ref{thm:main2}.} Let $G(A,B)$ be a matching-covered bipartite graph.

First, assume that $G(A,B)$ has a 2-edge-cut $S$ which separates $G(A,B)$ into two balanced components $G_1(A_1,B_1)$ and $G_2(A_2,B_2)$. Then $G_1(A_1,B_1)$ has an even number of vertices because $|A_1|=|B_1|$. Therefore, every perfect matching $M$ of $G(A,B)$ has an even number of edges of $S$. Hence, $|S\cap M|=0$ or $|S\cap M|=2$. In other words, $S\cap M=\emptyset$ or $S\subseteq M$. So $S$ is an equivalent class of $G(A,B)$.

In the following, assume that $G(A,B)$ has an equivalent class $K$. Let $e, e'\in K$. It suffices to show that $e$ is contained by a 2-edge-cut $S$ which separates $G(A,B)$ into two balanced components. Since $e$ is equivalent to $e'$, it follows that $G(A,B)\backslash e$ has no perfect matching containing $e'$.
Let $\mathcal G=\{Q_1,..,Q_k\}$ be the Dulmage-Mendelsohn Decomposition of $G(A,B)\backslash e$ and let $D$ be the Dulmage-Mendelsohn digraph. By Lemma~\ref{lem:DM-decom}, every $Q_i$ is matching-covered and $e'$ joins two vertices from different components, say $Q_i$ and $Q_j$ with $i\ne j$. Without loss of generality, assume that $q_iq_j$ is an arc of $D$  where $q_i$ and $q_j$ correspond to $Q_i$ and $Q_j$. By the definition of the Dulmage-Mendelsohn digraph, $e'$ joins a vertex of $Q_i\cap B$ and a vertex of
$Q_j\cap A$. \medskip

\noindent{\bf Claim:} {\sl The arc $q_iq_j$ is a cut-edge of $D$.} \medskip

\noindent{\it Proof of Claim.} If not, let $T$ be a directed cut containing $q_iq_j$. Then $T$ contains another arc, say $e''$. By Lemma~\ref{lem:acyclic}, $D$ is acyclic. Since $G(A,B)$ is matching-covered, by Theorem~\ref{thm:edge}, $D$ has one directed path $P$ intersecting all directed cuts. So $D$ has
exactly one source and one sink, say $q_1$ and $q_k$ respectively, where $q_1$ and $q_k$ correspond to $Q_1$ and $Q_k$. By Theorem~\ref{thm:edge}, adding an arc from $q_k$ to $q_1$ generates a strongly connected digraph $D'$. So there is a directed cycle $C$ containing $e''$. Note that $C$ contains exactly one arc in $T$. It follows that $C$ is still a directed cycle of $D'\backslash q_iq_j$. The directed cycle $C$ corresponds to an $M$-alternating cycle of $G(A,B)\backslash e'$ containing the edge $e$ for some perfect matching $M$ of $G(A,B)$. Therefore,
$G(A,B)$ has a perfect matching containing $e$ but not $e'$, contradicting that $e$ and $e'$ are equivalent to each other. This completes the proof of Claim. \medskip

By Claim, $E(Q_i, Q_j)\cup \{e\}$ is an edge-cut of $G(A,B)$. All edges in $E(Q_i,Q_j)\backslash \{e\}$ join a vertex of $Q_i\cap B$ and a vertex of $Q_j\cap A$. If $E(Q_i, Q_j)$ contains an edge $f$ other than $e$ and $e'$, then $G(A,B)\backslash e'$ is matching-covered because the Dulmage-Mendelsohn digraph of $G(A,B)\backslash \{e,e'\}$ is the same as $D$. Therefore, adding the edge $e$ makes $G(A,B)\backslash \{e'\}$ matching-covered. So $G(A,B)$ has a perfect matching containing $e$ but not $e'$, contradicting $e,e'\in K$ again.  The contradiction implies that $\{e, e'\}$ is a 2-edge-cut, which separates $G(A,B)$ into two components such that, for any $Q_m$ with $1\le m\le k$, a component of $G(A,B)\backslash \{e,e'\}$ either contains $Q_m$ or does not intersect $Q_m$. Hence, every component of $G(A,B)\backslash \{e, e'\}$ is balanced. This completes the proof. \qed

\medskip

\noindent{\bf Remark.} In \cite{CLM}, Carvalho et. al. proved that two equivalent edges $e$ and $e'$ of a matching-covered bipartite graph form an edge cut. Theorem~\ref{thm:main2} can be proved by the result of Carvalho et. al. easily. The proofs of Theorems~\ref{thm:main3} and \ref{thm:main2} in this paper are based on the Dulmage-Mendelsohn Decomposition which provides insight into the structure of matchable bipartite graphs.


\end{document}